\documentclass[12pt]{article}
\usepackage{epsfig,graphicx,color}
\usepackage{mathrsfs}

\usepackage{graphicx}
\usepackage{latexsym,amssymb}
\usepackage{amsthm}
\usepackage{indentfirst}
\usepackage{amsmath}
\usepackage{color}
\usepackage{xcolor}
\usepackage{fourier}
\usepackage[colorlinks=true,backref=page]{hyperref}

\newtheorem{theoremalph}{Theorem}

\newtheorem*{Main Theorem}{Main Theorem}

\newtheorem{Theorem}{Theorem}[section]
\newtheorem*{Theorem A}{Theorem A}
\newtheorem*{Theorem A'}{Theorem A'}
\newtheorem*{Theorem B'}{Theorem B'}
\newtheorem*{Theorem C'}{Theorem C'}
\newtheorem*{Theorem D'}{Theorem D'}

\newtheorem{Conj}{Conjecture}

\newtheorem{Proposition}[Theorem]{Proposition}
\newtheorem{Lemma}[Theorem]{Lemma}

\newtheorem{Remark}{Remark}

\newtheorem{Remark-numbered}{Remark}
\newtheorem{Corollary}[Theorem]{Corollary}

\newtheorem*{Claim}{Claim}
\newtheorem{Claim-numbered}{Claim}
\newtheorem*{Acknowledgements}{Acknowledgements}

 \def\NN{{\mathbb N}} 

 \def\RR{{\mathbb R}}

 \def\ZZ{{\mathbb Z}}

\def\La{\Lambda}

  \def\cG{{\cal G}} \def\cM{{\cal M}} \def\cS{{\cal S}}
  \def\cH{{\cal H}}  \def\cT{{\cal T}}
\def\cC{{\cal C}}    \def\cU{{\cal U}}
   \def\cP{{\cal P}} 
    
\def\cF{{\cal F}}   \def\cR{{\cal R}} \def\cX{{\cal X}}

\newcommand{\sing}{{\operatorname{Sing}}}
\newcommand{\diff}{{\operatorname{Diff}}}

\def\diff{\operatorname{Diff}}
\def\Int{\operatorname{Int}}

\def\ud{\operatorname{d}}
\def\e{{\varepsilon}}

\def\det{\operatorname{det}}

\def\wh{\widehat}

\begin{document}

\title{{Ergodic optimization for some dynamical systems beyond uniform hyperbolicity}}

\author{Dawei Yang and  Jinhua Zhang\footnote{  
D.Y  was partially supported by NSFC 11822109, 11671288, 11790274. J.Z was partially supported by starting grant from Beihang University.}}


\maketitle

\begin{abstract} 
In this paper, we show that  for several interesting systems beyond uniform hyperbolicity, any generic continuous function has a unique maximizing measure with zero entropy. In some cases, we also know that the maximizing measure has   full support. These interesting systems including singular hyperbolic attractors, $C^\infty$ surface diffeomorphisms and diffeomorphisms away from homoclinic tangencies. 

\hspace{-1cm}\mbox
\smallskip

\end{abstract}

\section{Introduction}
 
In this work, we will study the ergodic optimization problem for some vector fields and diffeomorphisms beyond uniform hyperbolicity. We can show that for singular hyperbolic attractors, $C^\infty$ surface diffeomorphisms and $C^1$ generic diffeomorphisms away from homoclinic tangencies, a generic continuous function has a unique maximizing measure and the maximizing measure has zero entropy.  In some cases, we can also give the information of the support of maximizing measure: the support is generally large.

This kind of results has been shown by Bousch \cite{Bou} and Br\'emont \cite{Br} for maps with some expansive property and specification property, which are satisfied by uniformly hyperbolic maps. We will give more details about the results of  Bousch and Br\'emont later.


\subsection{The abstract version}

 Let $(\phi_t)_{t\in\RR}$ be a continuous flow on a compact metric space $K$. Let us denote $C^0(K)$ the set of continuous functions on $K$, denote $\cM_{inv}(\phi_t)$ the set of invariant measures of  the flow $(\phi_t)_{t\in\RR}$ and denote $\cM_{erg}(\phi_t)$ the set of ergodic measures of the flow $(\phi_t)_{t\in\RR}$. Given   $\varphi\in C^0(K)$, a probability measure $\mu$ is a \emph{maximizing measure of $\varphi$}
 if it is invariant under $(\phi_t)_{t\in\RR}$ and 
 $$\beta(\varphi):=\sup_{\nu\in\cM_{inv}(\phi_t)}\int\varphi\ud\nu=\int\varphi\ud\mu.$$
Let us denote by $\cM_{max}(\varphi)$ the set of maximizing measures of 
$\varphi$. By definition, one can check that $\cM_{max}(\varphi)$ is   compact and convex, and if $\mu$ belongs to $\cM_{max}(\varphi)$, so do its ergodic components. Similar definitions can be made for diffeomorphisms in a parallel way. Hence they are omitted.

\subsection{Singular hyperbolic attractors}
The dynamics of vector fields differ from the one of diffeomorphisms and the presence of singularities is one of the reasons. As a prototype model,  
 Lorenz attractor \cite{Lo} has been substantially studied, see for instance~\cite{Gu, GuWi, Wi}. The singularities and periodic orbits in Lorenz attractor are hyperbolic. To make the hyperbolic structures of singularities and  periodic orbits compatible,  \cite{MPP} introduced the notion of ``singular hyperbolicity''.  

Let $M$ be a compact Riemannian manifold without boundary and let $\cX^1(M)$ be the space of $C^1$ vector fields endowed with $C^1$ topology. A compact invariant set $\La$ of a $C^1$-vector field $X$  is \emph{singular hyperbolic}, if there exist a  $(D\phi_t)_{t\in\RR}$-invariant splitting $T_\La M=E^{ss}\oplus E^{cu}$, where $(\phi_t)_{t\in\RR}$ is the flow generated by the vector field $X$,   and two numbers $C>1$ and $\lambda>0$ such that 
\begin{itemize} 
	\item \textbf{Contraction along $E^{ss}$:} $\|D\phi_t|_{E^{ss}(x)}\|\leq C\cdot e^{-\lambda t}$ for $x\in \La$ and $t\geq0$;
	\item \textbf{Domination:}	$\|D\phi_t|_{E^{ss}(x)}\|\cdot\|D\phi_{-t}|_{E^{cu}(\phi_t(x))}\|\leq C\cdot e^{-\lambda t}$ for $x\in \La$ and $t\geq0$;
	\item \textbf{Sectional expanding along $E^{cu}$:} for any $x\in\La$ and any 2-dimensional linear subspace $P\subset E^{cu}(x)$, one has 
	$$\det(D\phi_t|_P)\geq C^{-1}\cdot e^{\lambda t}~\textrm{ for   $t\geq0$}.$$
\end{itemize}
A compact invariant set $\La$ of a flow $(\phi_t)_{t\in\RR}$ is an attractor, if  
\begin{itemize}
	\item there exists a neighborhood $U$ of $\La$ such that $\phi_t(\overline{U})\subset \Int U$ for $t\geq 1$ and $\cap_{t\geq 0}\phi_t(U)=\Lambda$;
	\item $\La$ is a transitive set, i.e. $\La$ admits a dense orbit.
\end{itemize}
The Lorenz attractor has been shown to be singular hyperbolic, see for instance~\cite{MPP}.
It has been shown in~\cite[Corollary 1]{MSV} that for $C^1$-generic Lorenz attractor on 3-manifolds, for generic continuous function, the ergodic maximizing measure is unique with zero entropy and has full support on the Lorenz attractor. 
 In fact, one can show that such result holds in any dimension and the maximizing measure has zero entropy.  
 
 Recall that a property   is said to be $C^1$ generic for  vector fields if there is a dense $G_\delta$ set\footnote{A dense $G_\delta$ set is also said to be residual.} in $\cX^1(M)$ such that any vector field in this dense $G_\delta$ set has the property.

\begin{theoremalph}\label{thm.singular-hyperbolic}
Let $M$ be a $d$-dimensional manifold for $d\ge 3$. Let $X$ be a $C^1$ vector field on $M$ and $\Lambda$ be a singular hyperbolic attractor of $X$. Then there is a residual subset $\cR_\Lambda\subset C^0(\Lambda)$ such that any $\varphi\in \cR_\La$ has a unique maximizing measure  and the maximizing measure    has zero entropy. 
	
	Moreover, if $X$ is contained in some dense $G_\delta$ subset in $\cX^1(M)$, then the support of the maximizing measure of $\varphi\in\cR_\Lambda$ is $\Lambda$.
\end{theoremalph}

For flows with hyperbolicity, there are other works \cite{HLMXZ,MSV} which consider generic ergodic opmizations in higher regularities.

\subsection{$C^\infty$ surface diffeomorphisms}
Surfaces are lowest dimensional manifolds on which diffeomorphisms have complexity. There are many works on surface diffeomorphisms. See \cite{ABCD,B,Bu,CP, GG, F, HHTU, S} for a partial list. We also obtain some result on surface diffeomorphisms.

\begin{theoremalph}\label{Thm:surface-diff}
Assume that $f$ is a $C^\infty$   diffeomorphism on a surface $M$. Then there exists a residual subset  $\cR$ of $C^0(M)$ such that for each $\varphi\in \cR$, the maximizing measure of $\varphi$ is unique and has zero entropy.

\end{theoremalph}
Since the dynamics of surface diffeomorphisms may be decomposed into several pieces, we do not have the information on the support of ergodic maximizing measures.

\subsection{Diffeomorphisms away from homoclinic tangencies}

A diffeomorphism $f$ is said to have a \emph{homoclinic tangency} if $f$ has a hyperbolic periodic orbit $\gamma$ such that $W^s(\gamma)$ and $W^u(\gamma)$ intersect at some point non-transversely.    Palis \cite{Pa} conjectured that homoclinic tangencies is one of the obstructions to hyperbolicity, hence it is important to understand the dynamics far away from homoclinic tangencies, and many interesting results are obtained. One can see the introduction of \cite{CSY} to know the results in this direction. 

Recall that a compact  $f$-invariant set $\La$ is \emph{isolated} if there exists an open neighborhood $U$ of $\Lambda$ such that $\cap_{n\in\ZZ}f^n(U)=\La.$ We obtain the following  ergodic optimization results for diffeomorphisms away from homoclinic tangencies.

\begin{theoremalph}~\label{thm.homoclinic-class}
There exists a dense $G_\delta$ subset $\cG$ of $\diff^1(M)$ such that if $f\in\cG$   is away from homoclinic tangencies, then one has the following properties.
\begin{itemize}

\item There exists a residual subset of $\cR$ of $C^0(M)$ such that for each $\varphi\in \cR$, the maximizing measure of $\varphi$ is unique and has zero entropy.

\item If $\Lambda$ is an isolated transitive set, then there exists a residual subset of $\cR_\La$ of $C^0(\Lambda)$ such that for each $\varphi\in \cR_\La$, the maximizing measure of $\varphi$ is unique and has zero entropy. Moreover, the support of the maximizing measure is $\Lambda$.

\end{itemize}
%
%
\end{theoremalph}

In the direction of Theorem~\ref{thm.homoclinic-class}, we would like to formulate the following conjecture.

\begin{Conj}\label{con:with-singularity}
	Let $X$ be a $C^1$-generic vector field  on $M$ and $\Lambda$ be a homoclinic class which is away from homoclinic tangencies, and $\Lambda$ contains singularities.
	
	Then there exists a residual subset of $\cR$ of $C^0(\Lambda)$ such that for each $\varphi\in \cR$, the ergodic maximizing measure of $\varphi|_{\Lambda}$ is unique, has zero entropy, and  has $\Lambda$ as its support.
\end{Conj}

The main difficulty of Conjeture~\ref{con:with-singularity} comes from the upper semi continuity of the metric entropy. One can see some recent progress in~\cite{SYY}.

\subsection{A uniform mechanism}

For studying non-hyperbolic systems as in Theorems~\ref{thm.singular-hyperbolic}, ~\ref{Thm:surface-diff} and~\ref{thm.homoclinic-class}, we find some uniform mechanisms behind. We list several important properties for dynamical systems.
\begin{itemize}

\item A flow $(\phi_t)_{t\in\RR}$ on $K$ is said to have the \emph{property $\cP$} if each invariant measure of $(\phi_t)_{t\in\RR}$ is approximated by periodic orbits.

\item A flow $(\phi_t)_{t\in\RR}$ on $K$ is said to have the \emph{property $\cP_e$} if each \emph{ergodic} measure of $(\phi_t)_{t\in\RR}$ is approximated by periodic orbits. 

\item A flow $(\phi_t)_{t\in\RR}$ on $K$ is said to have the \emph{property $\cP_e^+$} if each \emph{ergodic} measure of $(\phi_t)_{t\in\RR}$ with positive metric entropy is approximated by periodic measures. 

\item A flow $(\phi_t)_{t\in\RR}$ on $K$ is said to have the \emph{property $\mathcal{S}$} if the metric entropy of $(\phi_t)_{t\in\RR}$ varies upper semi-continuously, that is, if a sequence $\{\mu_n\}$ of invariant measures of $(\phi_t)_{t\in\RR}$ converges to $\mu$, 
then $h_\mu(\phi_t)\ge\limsup_{n\to\infty}h_{\mu_n}(\phi_t)$.

\item A flow $(\phi_t)_{t\in\RR}$ on $K$ is said to have the \emph{property $\mathcal{F}_s$} if there is a dense $G_\delta$ set $\cM_R\subset \overline{\cM_{erg}(\phi_t)}$ such that the support of any invariant measure $\mu\in\cM_R$ is $K$.

\end{itemize}
It is clear by definition that the property $\cP$ is stronger that $\cP_e^+$.

\smallskip

Given a dynamical system, a generic continuous function has   a unique maximal measure. This was proved by Bousch \cite[Section 8]{Bou} for uniformly hyperbolic systems and by Br\'emont \cite{Br} for general dynamical systems\footnote{Br\'emont \cite{Br} gave a deep argument, but without precious statements that we will need.}. Under the conditions $\cP$ and $\cS$, by applying the arguments of Br\'emont, one knows that for generic continuous functions, the unique maximizing measure has  zero entropy. Under Bowen's specification property \cite{Bowen}, one knows that the maximizing measure has full support. Now we can generalize the results by weakening the property $\cP$ to $\cP_e^+$.

\begin{theoremalph}~\label{thm.abstract-result}
	Let $(\phi_t)_{t\in\RR}$ be a continuous flow on a compact metric space $K$. Then we have the following properties.
	
	\begin{enumerate}
	
	\item\label{i.uniqueness} There exists a residual subset $\cR_1$ of $C^0(K)$ such that for each $\varphi\in\cR_1$,  the set $\cM_{max}(\varphi)$ consists of a unique measure.
	\item\label{i.zero-etrnopy} If $(\phi_t)_{t\in\RR}$ has the properties $\cP_e^+$ and $\cS$,  then there exists a residual subset $\cR_2$ of $C^0(K)$ such that for each $\varphi\in\cR_2$, the set $\cM_{max}(\varphi)$ consists of a unique measure which has zero entropy.
	
	\item\label{i.full-support} If   $(\phi_t)_{t\in\RR}$ has the property $\cF_s$,    then there exists a residual subset $\cR_3$ of $C^0(K)$ such that for each $\varphi\in\cR_3$,  the set $\cM_{max}(\varphi)$ consists of a unique measure which has full support.
	
	\end{enumerate}
%
\end{theoremalph}
\begin{Remark}
We remark that the first item has been obtained in \cite{J1}. 
\end{Remark}
{We will apply Theorem~\ref{thm.abstract-result} to Theorem~\ref{thm.singular-hyperbolic} for $C^1$ generic singular hyperbolic vector field because one can have the periodic approximation property in this case.}

One has the following Theorem D' which is a homeomorphism version of Theorem~\ref{thm.abstract-result}. The proof will be the same, hence omitted.
\begin{Theorem D'}
	Let $h$ be a homeomorphism on a compact metric space $K$. Then we have the following properties.  	
	
	\begin{enumerate}
	
	\item\label{i.uniqueness-h} There exists a residual subset $\cR_1$ of $C^0(K)$ such that for each $\varphi\in\cR_1$, the set $\cM_{max}(\varphi)$ consists of a unique measure.
	
	\item\label{i.zero-etrnopy-h} If $h$ has the properties $\cP_e^+$ and $\cS$, then there exists a residual subset $\cR_2$ of $C^0(K)$ such that for each $\varphi\in\cR_2$, the set $\cM_{max}(\varphi)$ consists of a unique measure which has zero entropy.
	
	\item\label{i.full-support-h} If  $h$ has the property $\cF_s$, then there exists a residual subset $\cR_3$ of $C^0(K)$ such that for each $\varphi\in\cR_3$,  the set $\cM_{max}(\varphi)$ consists of a unique measure which has full support.
	
	\end{enumerate}
\end{Theorem D'}

Uniformly hyperbolic systems have the property $\cP$ and the property $\cF_s$ by Sigmund \cite{Si} and the property $\cS$ by Bowen~\cite{B72}. For uniformly hyperbolic diffeomorphisms, Theorem~\ref{thm.abstract-result} has been obtained, see for instance   \cite{Bou,Br}.
%
%

\section{The abstract version}

In this section, we collect the basic notions and results used in this paper.

\subsection{Some properties of maximizing measures}
 Recall that for any two probability measures $\mu$ and $\nu$ on a compact metric space, the distance between $\mu$ and $\nu$  can be defined as the following:
$$\ud(\mu,\nu)=\sum_{n=1}^\infty\frac{\big|\int \varphi_n {\rm d}\mu-\int \varphi_n {\rm d}\nu\big|}{2^n\cdot\|\varphi_n\|_{C^0}},$$
where $\{\varphi_n\}_{n\in\NN}$ is a dense subset of $C^0(K)$.
This gives the weak*-topology of the space of probability measures. 

We would like to give a list of simple properties of maximizing measures.

\begin{Proposition}\label{Pro:easy-property}
Assume $(\phi_t)_{t\in\RR}$ is a flow on $K$ and $\varphi$ is a continuous function on $K$. Then
\begin{enumerate}
\item\label{i.linear-comb} If $\mu\in \cM_{max}(\varphi)$, and there is $\lambda\in(0,1)$ such that $\mu=\lambda\cdot\mu_1+(1-\lambda)\cdot\mu_2$ for some invariant measures $\mu_1$ and $\mu_2$, then $\mu_i\in \cM_{max}(\varphi)$ for $i=1,2$.
\item\label{i.ergodic-component}If $\mu\in\cM_{max}(\varphi)$, then any ergodic component of $\mu$ is contained in $\cM_{max}(\varphi)$.
\item\label{i.linear-combination}Given two continuous functions $\varphi_1$ and $\varphi_2$, if $\mu\in\cM_{max}(\varphi_1)\cap\cM_{max}(\varphi_2)$, then for any $a,b\in\RR^+$, one has that $\mu\in\cM_{max}(a\varphi_1+b\varphi_2)$.
\item\label{i.usc-maximizing} The map $h\in C^0(K)\mapsto\cM_{max}(h)$ is upper semi-continuous.	

\end{enumerate}

\end{Proposition}
The proofs are almost direct. We give a quick proof of Item~\ref{i.usc-maximizing} of Proposition~\ref{Pro:easy-property} 
	\proof[Proof of Item~\ref{i.usc-maximizing} of Proposition~\ref{Pro:easy-property} ]
	Assume that there exist  a sequence of continuous functions $h_n$ converging to $h$ in $C^0$-topology and  a sequence of measures $\mu_n\in\cM_{max}(h_n)$ converging to a measure $\mu$ in weak$*$-topology.  Then for any invariant measure $\nu\in\cM_{inv}(\phi_t)$, one has the following
	\begin{align*}
	&\int h\ud\mu-\int h\ud\nu\\
	&=\big(\int h\ud\mu-\int h\ud\mu_n\big)+\big(\int h\ud\mu_n-\int h_n\ud\mu_n\big) +\big(\int h_n\ud\mu_n-\int h_n\ud\nu\big)+\big(\int h_n\ud\nu-\int h\ud\nu\big)
	\\
	&\geq \big(\int h\ud\mu-\int h\ud\mu_n\big)+\big(\int h\ud\mu_n-\int h_n\ud\mu_n\big) +\big(\int h_n\ud\nu-\int h\ud\nu\big).
	\end{align*}
	By taking the limit on the right hand of the inequality, one gets that
	$\int h\ud\mu\geq\int h\ud\nu$. By the arbitrariness of $\nu$, one has $\mu\in\cM_{max}(h),$ ending the proof of Item~\ref{i.usc-maximizing} of Proposition~\ref{Pro:easy-property}.
	\endproof
	A direct corollary of Item~\ref{i.usc-maximizing} of Proposition~\ref{Pro:easy-property} is the following corollary.
	\begin{Corollary}\label{Cor:continuous-point}
Let $K$ be a compact metric space and $(\phi_t)_{t\in\RR}$ be a continuous flow on 
$K$. Then there is a dense $G_\delta$ set $\widetilde\cF$ in $C^0(K)$ such that  the map $h\in C^0(K)\mapsto\cM_{max}(h)$ is continuous at any $h\in\widetilde\cF$.	

\end{Corollary}	
%
%
%
%
%
%

\subsection{The results by Br\'emont, Morris and Jenkinson}

\begin{Lemma}[Lemma 2.1 in \cite{M1}]~\label{l.decomposition}
	Let $K$ be a compact metric space and $(\phi_t)_{t\in\RR}$ be a continuous flow on 
	$K$. Consider two different probability measures   $\mu\in\cM_{inv}(\phi_t)$ and  $\nu\in\cM_{erg}(\phi_t)$. Assume that  there exists $c\in(0,2)$ such that 
	$$\big|\int\varphi\ud\mu-\int\varphi\ud\nu\big|\le c\cdot \|\varphi\|_{C^0} \textrm{ for any $\varphi \in C^0(K)$},$$
	where $\|\cdot\|_{C^0}$ is the $C^0$-norm on the space $C^0(K)$, i.e. $\|\varphi\|_{C^0}=\max_{x\in K} |\varphi(x)|.$ 
	
	Then there exist  $\lambda\in(0,1)$ and a probability measure $\wh\mu\in\cM_{inv}(\phi_t)$ such that $$\mu=\lambda\cdot\nu+(1-\lambda)\wh\mu.$$
	\end{Lemma}
	
The following result is obtained in \cite{Br} (see \cite{JM,PS} for similar results in other context), which shows that if an invariant measure almost optimizes a  continuous function, then one can perturb the continuous function and the measure such that the perturbed measure maximizes the perturbed function, to be precise:
\begin{Lemma}[\cite{Br}]~\label{l.perturbation-lemma}
	Let $K$ be a compact metric space and $(\phi_t)_{t\in\RR}$ be a continuous flow on 
$K$. Consider $\varphi\in C^0(K)$ and $\mu\in\cM_{inv}(\phi_t)$. If one has  $$\sup_{\nu\in\cM_{inv}(\phi_t)}\int\varphi\ud\nu=\int\varphi\ud\mu+\e \cdot\delta$$ for some small numbers $\e,\delta\geq 0$, then there exist $\wh\varphi\in C^0(K)$ and $\wh\mu\in\cM_{max}(\wh\varphi)$ such that
\begin{itemize}
	\item $\|\varphi-\wh\varphi\|_{C^0}\leq\e;$
	\item $|\int f\ud\mu-\int f\ud\wh\mu|\leq \delta\cdot\|f\|_{C^0}$ for any $f\in C^0(K)$.
\end{itemize}
\end{Lemma}

One can find the following lemma in \cite[Theorem 1]{J2}.
\begin{Lemma}\label{l.measure-maximize-unique-function}
	Let $K$ be a compact metric space and $(\phi_t)_{t\in\RR}$ be a continuous flow on 
	$K$. For any $\mu\in\cM_{erg}(\phi_t)$, there exists $\varphi\in C^0(K)$ such that $\cM_{max}(\varphi)=\{\mu \}.$
\end{Lemma}
{\begin{Remark}
The results in this subsection were stated for discrete systems, and one can extend such result to continuous systems by observing that each ergodic measure of a flow is  ergodic for  the time $t_0$-map of the flow (see~\cite{PuSh}) for some $t_0>0$ and an invariant measure of a flow is also invariant for  the time $t_0$-map of the flow. For example, we show how this works for Lemma~\ref{l.decomposition}. We take $t_0>0$ such that $\nu$ is also ergodic for the map $\phi_{t_0}$. Then one applies the result of Morris, there is $\lambda\in(0,1)$ and a $\phi_{t_0}$ invariant measure $\widehat\mu_0$ such that $\mu=\lambda\cdot\nu+(1-\lambda)\widehat\mu_0.$ Then it suffices to take 
$$\widehat\mu=\frac{1}{t_0}\int_0^{t_0}(\varphi_s)_*(\widehat\mu_0){\rm d}s.$$
\end{Remark}}
%
%

\subsection{The dual properties between measures and functions}
In this section, we give the proofs of our main results, and some ideas can be found in \cite{M1}.
\begin{Lemma}~\label{l.dense-function-with-unique-maximizing-measure}
Let $(\phi_t)_{t\in\RR}$ be a continuous flow on a compact metric space $K$. Let $\cU$ be a dense subset of $\cM_{erg}(\phi_t)$. Then the set 
	$$\cG:=\big\{h\in C^0(K): \textrm{$\cM_{max}(h)$ is unique and is contained in $\cU$} \big\}$$
	is dense in $C^0(K).$
\end{Lemma}
\proof
For any $h\in C^0(K),$ as $\cU$ is dense in $\cM_{erg}(\phi_t),$ for any $\e>0$,  there exists $\nu\in\cU$,  such that $\int h\ud\nu\geq \sup_{\mu\in\cM_{erg}}\int h\ud\mu-\e=\sup_{\mu\in\cM_{inv}}\int h\ud\mu-\e.$ 
 By Lemma~\ref{l.perturbation-lemma},  there exist  $h_\e\in C^0(K)$ and an invariant measure $\mu_\e\in\cM_{max}(h_\e)$ such that 
 \begin{itemize}
 	\item $\|h_\e-h\|_{C^0}\leq \e;$
 	\item $|\int\varphi\ud\nu-\int\varphi\ud\mu_\e|\leq \|\varphi\|_{C^0}$ for any $\varphi\in C^0(K)$.
 \end{itemize}  
 
 By Lemma~\ref{l.decomposition}, there exist  an invariant measure $\widehat{\mu_\varepsilon}$ and $\lambda\in(0,1)$ such that
 $$\mu_\varepsilon=\lambda\cdot\nu+(1-\lambda)\cdot{\widehat\mu}_\varepsilon.$$
 Thus $\nu$ is in the ergodic component of $\mu_\e$. By Item~\ref{i.ergodic-component} of Proposition~\ref{Pro:easy-property}, one has that $\nu\in\cM_{max}(h_\e)$. By Lemma~\ref{l.measure-maximize-unique-function}, there exists a continuous function $h^\prime$ such that $\cM_{max}(h^\prime)=\{\nu\}$. Consider the function $h_\e+\e\cdot \frac{h^\prime}{\|h^\prime\|_{C^0}}$ which is $2\e$-close to $h$. 
 \begin{Claim}
 $h_\e+\e\cdot \frac{h^\prime}{\|h^\prime\|_{C^0}}\in\cG.$
  \end{Claim}
 \begin{proof}[Proof of the Claim]
 By Item~\ref{i.linear-combination} of Proposition~\ref{Pro:easy-property} one has $\nu\in\cM_{max}(h_\e+\e\cdot \frac{h^\prime}{\|h^\prime\|_{C^0}})$. Assume by contradiction that there is another different invariant measure $\mu\in\cM_{max}(h_\e+\e\cdot \frac{h^\prime}{\|h^\prime\|_{C^0}})$. By the fact that $\cM_{max}(h^\prime)=\{\nu\}$, one has that
 $$\int \e\cdot \frac{h^\prime}{\|h^\prime\|_{C^0}}{\rm d}\mu=\frac{\varepsilon}{\|h'\|_{C^0}}\int h'{\rm d}\mu<\frac{\varepsilon}{\|h'\|_{C^0}}\int h'{\rm d}\nu.$$
 So we get a contradiction. By the fact $\nu\in\cU$, one gets that $h_\e+\e\cdot \frac{h^\prime}{\|h^\prime\|_{C^0}}\in\cG$ by definition.
 \end{proof} 
 One can conclude by the arbitrariness of $\varepsilon$.
 \endproof

\begin{Lemma}\label{l.open-dense-measure-give-open-dense-function}
Let $K$ be a compact metric space and $(\phi_t)_{t\in\RR}$ be a continuous flow on $K.$	Let $\cU$ be an open and dense subset of $\overline{\cM_{erg}(\phi_t)}$. Then the set 
$$\cF:=\big\{h\in C^0(K): \overline{\cM_{erg}(\phi_t)}\cap \cM_{max}(h)\subset\cU \big\}$$
is open and dense in $C^0(K).$
\end{Lemma}
\proof
Let $\cU_{erg}=\cU\cap\cM_{erg}(\phi_t)$. Since $\cU$ is open and dense in $\overline{\cM_{erg}(\phi_t)}$,  the set  $\cU_{erg}$ is dense in $\cM_{erg}(\phi_t)$. By Lemma~\ref{l.dense-function-with-unique-maximizing-measure}, the set
$$\cG_{erg}:=\big\{h\in C^0(K): \textrm{$\cM_{max}(h)$ is unique and is contained in $\cU_{erg}$} \big\}$$
is dense in $C^0(K)$. 

As $\overline{\cM_{erg}(\phi_t)}$ is compact and $\cU$ is open, by Item~\ref{i.usc-maximizing} of Proposition~\ref{Pro:easy-property}, for any $h\in\cF$, there is $\delta_h>0$ such that for any $h'\in B(h,\delta_h)$, one has $\overline{\cM_{erg}(\phi_t)}\cap\cM_{max}(h')\subset \cU$. Thus $\cF$ is open. 
%
\endproof

\subsection{The proof of Theorem~\ref{thm.abstract-result}}

Now, we give the proof of Theorem~\ref{thm.abstract-result}. Note that for a periodic orbit $\gamma$, we use $\delta_\gamma$ to denote the ergodic measure supported on $\gamma$.
\proof[Proof of Theorem~\ref{thm.abstract-result}]We prove the theorem by items.
\paragraph{The proof of Item~\ref{i.uniqueness}.}
It suffices to prove that for a  generic function $\varphi\in C^0(K)$, one has that $\#\big(\cM_{max}(\varphi)\cap\overline{\cM_{erg}(\phi_t)}\big)=1$.  We define a map 
\begin{eqnarray*}
\cM: \  C^0(K) & \to & \cC_{\cM_{inv}(\phi_t)} \\
    \varphi  & \mapsto & \cM_{max}(\varphi)\cap\overline{\cM_{erg}(\phi_t)},
\end{eqnarray*}
where $\cC_{\cM_{inv}(\phi_t)}$ denotes the space of all compact subsets of $\cM_{inv}(\phi_t)$ endowed with Hausdorff distance.
 By Corollary~\ref{Cor:continuous-point}, there is a residual set $\cR_1$ in $C^0(K)$ such that $\cM$ is continuous at any point in $\cR_1.$

 Now,  we prove that any map $\varphi\in\cR_1$ has a unique maximizing measure. Assume, on the contrary,  that there is $\varphi_0\in\cR_1$ such that $\varphi_0$ has at least two maximizing measures. Thus, by Lemma~\ref{l.dense-function-with-unique-maximizing-measure}, there is a sequence of $\{\varphi_n\}$ such that $\lim_{n\to\infty}\varphi_n=\varphi_0$ and each $\varphi_n$ has a unique maximizing measure.
Since  the map $\cM$ is continuous at $\varphi_0$, one has that $\lim_{n\to\infty}\cM_{max}(\varphi_n)=\cM_{max}(\varphi_0)$. But $\#\cM_{max}(\varphi_0)\ge 2$, hence $\{\cM_{max}(\varphi_n)\}$ will not converge. This gives a contradiction.

\paragraph{The proof of Item~\ref{i.zero-etrnopy}.}We consider the map
\begin{eqnarray*}
\cH: \  \overline{\cM_{erg}(\phi_t)} & \to & \RR^+\cup\{0\} \\
    \mu  & \mapsto & h_{\mu}(\phi_t).
\end{eqnarray*}
By the property $\cS$, the map $\cH$ is upper-semi continuous. Thus, there is a residual set $\cC$ in $\overline{\cM_{erg}(\phi_t)}$ such that any point in $\cC$ is a continuous point of $\cH$. Thus, one can assume that there exists a sequence of open and dense subsets $\cU_n$ in $\overline{\cM_{erg}(\phi_t)}$ such that $\cC=\cap_{n\in\NN}\cU_n$. Let 
$$\cF_n:=\big\{h\in C^0(K): \textrm{$\overline{\cM_{erg}(\phi_t)} \cap\cM_{max}(h)$ is contained in $\cU_n$}\big\}.$$
By Lemma~\ref{l.open-dense-measure-give-open-dense-function}, one has that $\cF_n$ is open and dense in $C^0(K)$. Now we define 
$$\cR_2=\cR_1\cap(\bigcap_{n\in\NN}\cF_n).$$
We have already proved that any $\varphi\in\cR_2$ has a unique maximizing measure which is ergodic. Denote by $\{\mu\}=\cM_{max}(\varphi)$.
\begin{Claim}
$\mu$ is contained in $\cap_{n\in\NN}\cU_n.$
\end{Claim}
\begin{proof}
For proving $\mu\in \cap_{n\in\NN}\cU_n$, it suffices  to prove that $\mu\in\cU_n$ for any $n\in\NN$. Since $\{\mu\}=\cM_{max}(\varphi)$ and $\varphi\in\cF_n$, by the definition of $\cF_n$, one has that $\mu\in\cU_n$.
\end{proof}
Now we show that $\mu$ has zero metric entropy. Otherwise, by the property $\cP_e^+$, $\mu$ is accumulated by periodic ergodic measures $\{\delta_{\gamma_n}\}$ for some periodic orbits $\{\gamma_n\}$. Since any element in $\cap_{n\in\NN}\cU_n$ is the continuous point of the map $\nu\mapsto h_\nu(\phi_t)$, one has that
$$h_\mu(\phi_t)=\lim_{n\to\infty}h_{\delta_{\gamma_n}}(\phi_t)=0.$$
Thus one gets a contradiction.
\paragraph{The proof of Item~\ref{i.full-support}.}We consider the map
\begin{eqnarray*}
\cT: \  \overline{\cM_{erg}(\phi_t)} & \to & \cC_K \\
    \mu  & \mapsto & {\rm supp}(\mu),
\end{eqnarray*}
where $\cC_K$ denotes the space of compact subsets of $K$ endowed with Hausdorff topology.
The map $\cT$ is lower semi-continuous. Hence there is a residual set $\cR_\cT$ in $\overline{\cM_{erg}(\phi_t)}$ such that any point in $\cR_\cT$ is a continuous point of $\cT$. Thus, one can assume that there exists a sequence of open and dense subsets $\cU'_n$ in $\overline{\cM_{erg}(\phi_t)}$ such that $\cT=\cap_{n\in\NN}\cU'_n$. Let 
$$\cF'_n:=\big\{h\in C^0(K): \textrm{$\cM_{max}(h)$ is contained in $\cU'_n$}\big\}.$$
By Lemma~\ref{l.open-dense-measure-give-open-dense-function}, one has that $\cF'_n$ is open and dense in $C^0(K)$.  Now we define 
$$\cR_3=\cR_1\cap(\bigcap_{n\in\NN}\cF'_n)$$
We have already proved that any $\varphi\in\cR_1$ has a unique maximizing measure which is ergodic. Denote by $\{\mu\}=\cM_{max}(\varphi)$.
\begin{Claim}
$\mu$ is contained in $\cap_{n\in\NN}\cU_n'$.
\end{Claim}
\begin{proof}
For proving $\mu\in \cap_{n\in\NN}\cU'_n$, we only need  to prove that $\mu\in\cU'_n$ for any $n\in\NN$. Since $\{\mu\}=\cM_{max}(\varphi)$ and $\varphi\in\cF'_n$, by the definition of $\cF'_n$, one has that $\mu\in\cU'_n$.
\end{proof}
Now we show that $\mu$ has full support.  Since generic measures in $\overline{\cM_{erg}(\phi_t)}$ have  full support,  
there is a sequence of invariant measures $\{\mu_n\}\subset \overline{\cM_{erg}(\phi_t)}$ such that $\mu=\lim_{n\to\infty}\mu_n$ and each $\mu_n$ has full support. Since $\mu\in\cT$, one has that
$${\rm supp}(\mu)=\lim_{n\to\infty}{\rm supp}(\mu_n)=K.$$
\endproof
%
%
%
%
%
%
%
%
%

\section{Generic measures of generic systems}

We will prove the following theorem in this section.

\begin{Theorem}\label{Thm:generic-measure}
There is a dense $G_\delta$ set $\cG$ in $\cX^1(M)$ such that for any $X\in\cG$, if $\Lambda$ is a non-trivial isolated transitive set, then
\begin{enumerate}
\item periodic measures are dense in $\cM_{inv}(\Lambda)$, in particular, $\overline{\cM_{erg}(\Lambda)}=\cM_{inv}(\Lambda)$;
\item there is a dense $G_\delta$ set $R_\Lambda$ in $\cM_{inv}(\Lambda)$ such that  ${\rm supp}(\mu)=\Lambda$ for any $\mu\in R_\Lambda$.
\end{enumerate}
\end{Theorem}
Theorem~\ref{Thm:generic-measure} is the (non-trivial) flow version of Item (a) and Item (b) of Theorem 3.5 in \cite{ABC}. However, here we have to consider the singular measures in $\Lambda$ which can also be accumulated by periodic measures by using the $C^1$ connecting lemma by \cite{WX} and \cite{BC}. Moreover, in the proof of \cite[Theorem 3.5]{ABC}, they can prove that any ergodic measure can be accumulated by a sequence of periodic measures such that its support is also accumulated \cite[Theorem 4.1]{ABC}. In our case, the ergodic measure supported on a singularity can never be accumulated by the support.

\medskip

Recall that given a vector field $X$, a point $\sigma$ is said to be a \emph{singularity} of $X$ if $X(\sigma)=0$ and we denote by $\sing(X)$ the set of singularities of $X$. The Dirac measure supported on a singularity $\sigma$ is denoted by $\delta_\sigma$.

\subsection{Connecting lemma and its consequence}
Recall that for a vector field $X$, a transitive compact invariant set $\Lambda$ is said to be \emph{non-trivial} if it is not reduced to  a singularity or a periodic orbit.
\begin{Proposition}\label{Pro:singular-measure-perturbation}

Let $\Lambda$ be a non-trivial transitive set of a $C^1$ vector field $X$ and $\sigma\in\Lambda$ be a hyperbolic singularity. Then for any $\varepsilon>0$, there are a vector field $Y$ and a periodic orbit $\gamma$ of $Y$ such that
\begin{itemize}
\item $Y$ is $\varepsilon$ $C^1$-close to $X$,

\item the periodic measure $\delta_\gamma$ is $\varepsilon$ close to the singular measure $\delta_\sigma$.

\end{itemize}

\end{Proposition}

\begin{Remark}
In general, we will not have that the support of $\delta_\gamma$ is close to the support of $\delta_\sigma$.
\end{Remark}

A direct corollary of Proposition~\ref{Pro:singular-measure-perturbation} is the following.
\begin{Corollary}\label{Cor:generic-accumulation}
There is a dense $G_\delta$ set $\cG\subset \cX^1(M)$ such that for any $X\in\cG$,  if a  singularity $\sigma$ of $X\in\cG$ is contained in a non-trivial  transitive set of $X$, then there is a sequence of periodic orbits $\gamma_n$ such that 
$$\delta_\sigma=\lim_{n\to\infty}\delta_{\gamma_n}.$$
\end{Corollary}

The proof of Proposition~\ref{Pro:singular-measure-perturbation} is based on the \emph{$C^1$ connecting lemma} original from Hayashi \cite{H}. We use the following version from Wen-Xia \cite{WX}.
\begin{Lemma}\label{Lem:connecting}
Let  $X$ be  a $C^1$ vector field. Given any point $z$ which is neither a singularity nor a periodic point, for any neighborhood $\cU$ of $X$, there is $L>0$ such that for any neighborhood $U_z$ of $z$, there is a neighborhood $V_z\subset U_z$ of $z$ with the following properties.

For any points $x$ and $y$, if
\begin{itemize}
\item $x,y\notin \phi_{[0,L]}(U_z)$;
\item the forward orbit of $x$ enters $V_z$ and the backward orbit of $y$ enters $V_z$;

\end{itemize}
then there exist a vector field $Y\in\cU$ and a time $T>0$ such that
\begin{itemize}
\item $\phi_T^Y(x)=y$,
\item $Y(p)=X(p)$ for any $p\in M\setminus \phi_{[0,L]}(U_z)$.
\end{itemize}
\end{Lemma}

\begin{Lemma}\label{Lem:homoclinic-orbit}

Let $\Lambda$ be a non-trivial transitive set of a vector field $X$. Let $\sigma\in\Lambda$ be a hyperbolic singularity. Then for any $\varepsilon>0$, there are a vector field $Y$ and a periodic orbit $\gamma$ of $Y$ such that
\begin{itemize}
\item $Y$ is $\varepsilon$ $C^1$-close to $X$,

\item $\sigma_Y$ has a homoclinic orbit.

\end{itemize}

\end{Lemma}

\begin{proof}
We choose a point $p\in\Lambda$ such that $\omega(p)=\alpha(p)=\Lambda$ by the fact that $\Lambda$ is transitive. We then choose two points $p^s\in W^s(\sigma)\cap\Lambda\setminus\{\sigma\}$ and $p^u\in W^s(\sigma)\cap\Lambda\setminus\{\sigma\}$.

Given $\varepsilon>0$, for $p^s$ and $p^u$, by  Lemma~\ref{Lem:connecting}, we choose a constant $L>0$ large enough. Then we choose very small neighborhood $U_{p^s}$ and $U_{p^u}$ such that
\begin{itemize}
\item $p,\sigma\notin \phi_{[0,L]}(U_{p^s})$,
\item $p,\sigma\notin \phi_{[-L,0]}(U_{p^u})$,
\item $\phi_{[0,L]}(U_{p^s})\cap \phi_{[-L,0]}(U_{p^u})=\emptyset$.
\end{itemize}
Now we can have smaller neighborhoods $V_{p^s}$ and $V_{p^u}$ as in Lemma~\ref{Lem:connecting}. Since the forward and backward orbits of $p$ will enter $V_{p^s}$ and $V_{p^u}$ respectively, by applying the connecting lemma (Lemma~\ref{Lem:connecting}), one gets that $p$ is a homoclinic point of $\sigma$ for some $Y$ that is $\varepsilon$ $C^1$-close to $X$.
\end{proof}

\begin{proof}[Proof of Proposition~\ref{Pro:singular-measure-perturbation}]
By Lemma~\ref{Lem:homoclinic-orbit}, for any $\varepsilon>0$, there is a vector field $Y$ such that $Y$ is $\varepsilon/2$-close to $X$ in $C^1$-topology and $\sigma$ has a homoclinic orbit associated to $Y$. For the homoclinic orbit, by a simple and standard perturbation, one gets that there is a sequence of vector fields $\{Y_n\}$ such that each $Y_n$ has a periodic orbit $\gamma_n$ such that
\begin{itemize}
\item $\lim_{n\to\infty}Y_n=Y$,
\item the Hausdorff limit of $\{\gamma_n\}$ is the homoclinic orbit.
\end{itemize}
Since the unique ergodic measure supported on the homoclinic orbit is $\delta_\sigma$, one has that $\lim_{n\to\infty}\delta_{\gamma_n}=\delta_\sigma$.
\end{proof}

\subsection{The properties of homoclinic classes}
Given an open set $U$ and a hyperbolic periodic orbit $\gamma$ of $X$, one can define the local homoclinic class 
$$H_U(\gamma)={\rm Closure}\big\{x\in W^s(\gamma)\pitchfork W^u(\gamma):~{\rm Orb}(x)\subset U\big\}.$$
 Recall that two hyperbolic periodic orbits are \emph{homoclinically related}, if the stable manifold of one periodic orbit has non-empty transverse intersection with unstable manifold of the other periodic orbit, and vice versa.  

For a local homoclinic class $H_U(\gamma)$, one has the following property:

\begin{Lemma}\label{Lem:measure-approximation}
For a homoclinic class $H_U(\gamma)$, there is a sequence of periodic orbits $\{\gamma_n\}\subset U$ such that
\begin{itemize}

\item Each $\gamma_n$ is homoclinically related with $\gamma$.

\item $\lim_{n\to\infty}\gamma_n=H_U(\gamma)$.

\item $\lim_{n\to\infty}\delta_{\gamma_n}=\delta_\gamma$.

\end{itemize}

\end{Lemma}
\begin{proof}
By the definition of a local homoclinic class,  there is a compact invariant set $\Gamma$ such that
\begin{itemize}
\item $\Gamma$ is the union of $\gamma$ and finitely many transverse homoclinic orbits of $\gamma$ in $U$.
\item $\Gamma$ is very close to $H_U(\gamma)$ in the Hausdorff distance.
\end{itemize}
We know that $\Gamma$ is a chain transitive hyperbolic set and the unique invariant measure on $\Gamma$ is $\delta_\gamma$. By applying the Shadowing lemma, there is a sequence of periodic orbits $\gamma_n$ such that
$\lim_{n\to\infty}\gamma_n=\Gamma$ and each $\gamma_n$ is homoclinically related with $\gamma$. Since the unique invariant measure on $\Gamma$ is $\delta_{\gamma}$, one has that $\lim_{n\to\infty}\delta_{\gamma_n}=\delta_\gamma$. One can conclude by increasing $\Gamma$.
\end{proof}

By a well-known result in \cite{BC}, one has
\begin{Proposition}\label{Pro:generic}
There is a dense $G_\delta$ set $\cG$ in $\cX^1(M)$ such that for any $X\in\cG$,  if $\Lambda$ is a non-trivial isolated transitive set, then $\Lambda$ is a local homoclinic class $H(\gamma)$.
\end{Proposition}

\subsection{Ergodic closing lemma and its consequence}

R. Ma\~n\'e \cite{Ma} proved the celebrated ergodic closing lemma. One has the following version for vector fields proved by L. Wen \cite{Wen}.

\begin{Lemma}\label{Lemma:measure-accumulation}
Assume that $\mu$ is an ergodic non-singular measure of a vector field $X$. Then for any $\varepsilon>0$, there are a vector field $Y$ and a periodic orbit $\gamma$ of $Y$ such that
\begin{itemize}
\item $Y$ is $\varepsilon$ close to $X$;

\item the periodic measure $\delta_\gamma$ is $\varepsilon$ close to $\mu$;

\item The support of $\mu$ is $\varepsilon$-close to $\gamma$ in the Hausdorff distance.

\end{itemize}

\end{Lemma}
An invariant measure $\mu$ is \emph{non-singular} if $\mu(\sing(X))=0$. As a   corollary of Lemma~\ref{Lemma:measure-accumulation}, one has the following result. 
\begin{Corollary}\label{Cor:non-singular-periodic-accumulation}
There is a dense $G_\delta$ subset $\cG\subset \cX^1(M)$ such that for any vector field $X\in\cG$, any non-singular ergodic measure $\mu$ is accumulated by a sequence of periodic measures in the weak*-topology.

\end{Corollary}


Given finitely many invariant measures $\mu_1,\cdots,\mu_k$ and   non-negative numbers $\lambda_1,\cdots,\lambda_k$ satisfying $\sum_{i=1}^k\lambda_i=1$, the invariant measure $\sum_{i=1}^k\lambda_i\cdot\mu_k$ is said to be a \emph{finite convex combination} of $\{\mu_i\}_{i=1}^k$.

%
%
%

%

\subsection{The proof of Theorem~\ref{Thm:generic-measure}}

We have the following version of \cite[Theorem 3.10]{ABC}
\begin{Theorem}\label{Thm:combination-periodic}
There is a dense $G_\delta$ subset $\cG\subset \cX^1(M)$ such that for any vector field $X\in\cG$, for any isolated transitive set $\Lambda$ of $X$, any finite convex combination of periodic measures can be accumulated by a sequence of periodic measures.

\end{Theorem}

We do not give the proof of Theorem~\ref{Thm:combination-periodic} because it only concerns the regular part of the dynamics and has nothing to do with singularities.

Now we can give the proof of Theorem~\ref{Thm:generic-measure}.
\begin{proof}[Proof of Theorem~\ref{Thm:generic-measure}]
We take the dense $G_\delta$ subset $\cG$ satisfying the properties  in Corollary~\ref{Cor:generic-accumulation}, Corollary~\ref{Cor:non-singular-periodic-accumulation} and Theorem~\ref{Thm:combination-periodic}

\smallskip

Now we assume that  $X\in\cG$ and $\Lambda$ is a non-trivial isolated transitive set.

\paragraph{The periodic accumulation.} By the ergodic decomposition theorem, as observed by \cite[Page 19]{ABC}, any invariant measure $\mu$ can be accumulated by finite convex combinations of ergodic measures. One has the following properties.
\begin{itemize}
\item Since $\Lambda$ is a non-trivial isolated transitive set, any singularity in $\Lambda$ accumulated by recurrent regular points. Thus, by Corollary~\ref{Cor:generic-accumulation}, any singular measure is accumulated by periodic measures.
\item By Corollary~\ref{Cor:non-singular-periodic-accumulation}, any non-singular ergodic measure can be accumulated by periodic measures, and such periodic measures are supported on $\Lambda$ due to the isolation property of $\Lambda.$
\end{itemize}
Thus, any invariant measure $\mu$ can be accumulated by finite convex combinations of periodic measures. By Theorem~\ref{Thm:combination-periodic}, any finite convex combination
of periodic measures can be accumulated by periodic measures. 

Combining with  the results above, one knows that any invariant measure on $\Lambda$ can be accumulated by periodic measures.

\paragraph{Full support.} Now we show that generic measures have full support.  We use $\cC_M$ to denote the space of compact subsets of $M$ endowed with Hausdorff topology.  Since  the map
\begin{eqnarray*}
\cT: \  \cM_{inv}(\phi_t) & \to &  \cC_M \\
    \mu  & \mapsto & {\rm supp}(\mu)
\end{eqnarray*}
is lower semi-continuous, hence  generic measures are continuous points of $\cT$. Let $\mu$ be such kind of generic point. By the first item, one has that
$$\mu=\lim_{n\to\infty}\delta_{\gamma_n}.$$
By Lemma~\ref{Lem:measure-approximation}, there is a sequence of periodic orbits $\gamma_n'$ such that
\begin{itemize}

\item $\gamma_n$ is homoclinically related with $\gamma_n'$.

\item $\lim_{n\to\infty}\gamma_n'=\Lambda$. In other words, $\lim_{n\to\infty}{\rm supp}(\delta_{\gamma_n'})=\Lambda$.

\item The distance between $\delta_{\gamma_n}$ and $\delta_{\gamma_n'}$ is less than $1/n$.

\end{itemize}
Thus, we have that
$${\rm supp}(\mu)=\lim_{n\to\infty}{\rm supp}(\delta_{\gamma_n'})=\Lambda.$$
\end{proof}

\section{The proof of the main theorems}

\subsection{Singular hyperbolic attractors}
For proving Theorem~\ref{thm.singular-hyperbolic}, we have to verify the properties $\cP_e^+$ and $\cS$.

\begin{Theorem}\label{Thm:singular-usc}
Any singular hyperbolic attractor $\Lambda$ has the properties $\cS$ and $\cP_e^+$.
\end{Theorem}
\begin{proof}
It has been proved in \cite{PYY} that any singular hyperbolic attractor is entropy expansive. Then a classical work of Bowen \cite{B72} implies such system has the property $\cS$.

Now we consider an ergodic measure $\mu$, which is supported on a singular hyperbolic attractor and  has positive entropy. Since the entropy of $\mu$ is positive, we know that $\mu$ cannot be supported on singularities. Thus, $\mu$ is a regular hyperbolic measure. By applying a shadowing lemma of Liao \cite{L85}, as stated in \cite{GY}, it has been shown in \cite[Page 214]{SGW} that $\mu$ is accumulated by  periodic measures. One can also see \cite{WYZ} for some details.

\end{proof}

Now Theorem~\ref{thm.singular-hyperbolic} follows from Theorem~\ref{thm.abstract-result}, Theorem~\ref{Thm:singular-usc} and Theorem~\ref{Thm:generic-measure} directly.

\subsection{$C^\infty$ surface diffeomorphisms}
We are going to prove Theorem~\ref{Thm:surface-diff}. We have the following result from Yomdin \cite{Y} and Newhouse \cite{N}.

\begin{Theorem}\label{Thm:surface-expansive}
Any $C^\infty$ diffeomorphism has the property $\cS$.
\end{Theorem}
The following result is also well known. We give a sketch.

\begin{Theorem}\label{Thm:surface-positive}
Any $C^\infty$ surface diffeomorphism  has the property $\cP_e^+$.
\end{Theorem}
\begin{proof}
Assume that $f$ is a surface diffeomorphism and $\mu$ is an ergodic measure with positive entropy. As in Katok \cite{K}, by applying Ruelle's inequality \cite{R}, one knows that $\mu$ is hyperbolic. Then it was proved in \cite{K} that  $\mu$ is accumulated by periodic orbits. See also \cite{G}.
\end{proof}
Theorem~\ref{Thm:surface-diff} follows from Theorem D', Theorem~\ref{Thm:surface-expansive} 
and Theorem~\ref{Thm:surface-positive} directly.

\subsection{Diffeomorphisms away from homoclinic tangencies}

\begin{proof}[Proof of Theorem~\ref{thm.homoclinic-class}] It has been proved in \cite{LVY} that diffeomorphisms away from homoclinic tangencies have the property $\cS$. See also \cite{CSY} and \cite{DFPV} for a generic version. Then it follows that for $C^1$ generic diffeomorphisms, any ergodic measure can be accumulated by periodic measures by using Ma\~n\'e's ergodic closing lemma \cite{Ma}. This verifies property $\cP_e.$ Now, the first item of Theorem~\ref{thm.homoclinic-class} follows from Theorem D'.
	
Given  a non-trivial isolated transitive set $\Lambda$,  by Theorem 3.5 in \cite{ABC}, 	one has that 
\begin{itemize}
	\item $\overline{\cM_{erg}(\Lambda)}=\cM_{inv}(\La)$;
	\item  generic measure $\mu$ on $\Lambda$ has full support, in formula: ${\rm supp}(\mu)=\La$.
\end{itemize}
	Thus, the proof of   second item is complete by Theorem D'.
\end{proof}

\begin{Acknowledgements}
We would like to thank P. Varandas for his nice comments.
\end{Acknowledgements}

%
%
%
%
%
%
%
%


\vskip 5pt

\begin{tabular}{l l l}
\emph{\normalsize Dawei Yang}
& \quad \quad&
\emph{\normalsize Jinhua Zhang}
\medskip\\

\small School of Mathematical Sciences
&& \small  School of Mathematical Sciences\\
\small Soochow University
&& \small Beihang University\\
\small Suzhou, 215006, P.R. China
&& \small Beijing 100191, P.R. China.\\
\texttt{yangdw1981@gmail.com}
&&\texttt{zjh200889@gmail.com}\\
\texttt{yangdw@suda.edu.cn}
&&\texttt{jinhua$\_$zhang@buaa.edu.cn}
\end{tabular}

\end{document}